\documentclass[11pt]{article}
\usepackage{lmodern}

\usepackage[T1]{fontenc}
\usepackage[utf8]{inputenc}
\usepackage[a4paper]{geometry}
\usepackage{color}
\usepackage{fancybox}
\usepackage{calc}
\usepackage{url}
\usepackage{bm}
\usepackage{amscd}
\usepackage{amsthm}
\usepackage{amsmath}
\usepackage{amssymb}
\usepackage{fixltx2e}
\usepackage{setspace}
\usepackage[unicode=true,pdfusetitle,
 bookmarks=false,bookmarksnumbered=false,bookmarksopen=false,
 breaklinks=false,pdfborder={0 0 1},backref=section,colorlinks=true]
 {hyperref}
\hypersetup{
 linkcolor=MidnightBlue,citecolor=RubineRed}
\usepackage{breakurl}

\makeatletter
\numberwithin{equation}{section}
  \theoremstyle{plain}
  \newtheorem{lem}{\protect\lemmaname}[section]
  \theoremstyle{plain}
  \newtheorem{thm}{\protect\theoremname}[section]
  \theoremstyle{plain}
  \newtheorem{fact}{\protect\factname}[section]

\usepackage[usenames,dvipsnames]{xcolor}\usepackage{chemfig}\usepackage{mathtools}\usepackage{amsthm}\usepackage{fancybox}

\usepackage{bm}\usepackage{url}\usepackage{array}\usepackage{titlesec}\usepackage{pifont}\usepackage{mathrsfs}\usepackage{yfonts}\usepackage{stmaryrd}\usepackage{wasysym}%a lot of packages!
\usepackage{bigints}\DeclareMathAlphabet{\mathpzc}{OT1}{pzc}{m}{it}

\usepackage[tikz,blur]{bclogo}%the Cool colourful boxes!
\usepackage[parfill]{parskip}% colourful clickable links!
\usepackage[all]{hypcap}

\renewcommand{\C}{\mathbb{C}}
\newcommand{\su}{\mathfrak{su}}
\renewcommand{\u}{\mathfrak{u}}
\renewcommand{\sp}{\mathfrak{sp}}
\newcommand{\R}{\mathbb{R}}

\titleformat{\subsection}[block]{\scshape\centering}{\thesubsection}{1em}{} 
\titleformat{\section}[block]{\bfseries\Large}{\thesection}{1em}{} 
\numberwithin{equation}{section}

\newcommand{\norm}[1]{\left|\left|#1\right|\right|}

\newcommand{\paren}[1]{\left(#1\right)}
\newcommand{\abrace}[1]{\left\langle#1\right\rangle}

\renewcommand{\epsilon}{\varepsilon}
\renewcommand{\ker}[1]{\operatorname{ker}(#1)}

\newcommand{\ve}[1]{\mathbf{#1}}
\newcommand{\bvec}[1]{\vec{\bm{#1}}}

\newcommand{\actson}{\curvearrowright}

\newtheorem{theorem}{Theorem}[section]

\newtheorem{defn}{Definition}

\numberwithin{defn}{section}
\definecolor{bgblue}{RGB}{254,254,251}
\DeclareMathOperator{\Isoc}{Isoc}
\DeclareMathOperator{\Tr}{Tr}
\DeclareMathOperator{\Skew}{Skew}

\title{\vspace{-1truecm}\hrulefill \\[0.1cm] \huge The Lie-Poisson Structure of the Symmetry Reduced Regularised n-Body Problem \\[-0.4cm]\hrulefill\\[-0.2cm]}
\author{Suntharan Arunasalam and Holger R.{} Dullin and Diana M.H. Nguyen}
\date{\today}

\AtBeginDocument{

}

  \providecommand{\factname}{\inputencoding{latin9}Fact}
  \providecommand{\lemmaname}{\inputencoding{latin9}Lemma}
\providecommand{\theoremname}{\inputencoding{latin9}Theorem}

\makeatother

  \providecommand{\factname}{Fact}
  \providecommand{\lemmaname}{Lemma}
\providecommand{\theoremname}{Theorem}

\begin{document}
\maketitle \vspace{-0.6cm}

\begin{abstract}
\noindent This paper investigates the symmetry reduction of the regularised n-body problem. The three body problem, regularised through quaternions, is examined in detail. We show that for a suitably chosen symmetry group action the space of quadratic invariants is closed and the Hamiltonian can be written in terms of the quadratic invariants. The corresponding Lie-Poisson structure is isomorphic to the Lie algebra $\mathfrak{u}(3,3)$. Finally, we generalise this result to the n-body problem for $n>3$. 
\end{abstract}

\section{Introduction}

The  Galilean symmetry of the $n$-body problem  leads
to the classical 9 integrals of linear momentum, centre of mass, and
angular momentum. Symplectic reduction of this symmetry gives a
 reduced system with $3n-5$ degrees of freedom, see, e.g.~\cite{MeyerHall}.
An alternative approach to reduction that avoids problems with singular
reduction uses invariants of the symmetry group action. Singular reduction
does occur in the $n$-body problem because the orbit of the symmetry
group drops in dimension for collinear configurations. Using quadratic
invariants leads to a Lie-Poisson structure isomorphic to $\sp(2n-2)$,
as was shown using different bases of invariants in \cite{Sadetov02}
and \cite{Dullin13}.

One motivation for this approach is the possibility to derive structure
preserving geometric integrators for the symmetry reduced $3$-body
problem, as done in \cite{Dullin13}. However, numerical integration
of many body problems needs to be able to deal with binary near-collisions.
The classical regularisation by squaring in the complex plane found
by Levi-Civita \cite{Levi-Civita} has a beautiful spatial analogue
that can be formulated using quaternions \cite{KS}, also see \cite{Saha}. This regularisation
has been used by Heggie to simultaneously regularise binary collision
in the $n$-body problem \cite{Heggie}. Recently the symmetry reduction
of the regularised 3-body problem has been revisited in \cite{mont},
extending the classical work of Lemaitre \cite{Lemaitre}. In the present work we
perform the symmetry reduction using quadratic invariants, thus repeating
\cite{Dullin13} for the regularised problem. 
See \cite{LMS} for some background on singular reduction.
Our main result is that
the symmetry reduced regularised 3-body problem has the Lie-Poisson
structure of the Lie-algebra $\u(3,3)$.

The paper is organised as follows. In the next section we introduce
our notation of quaternions and Heggie's regularised Hamiltonian. We
then treat the cases $n=2$ (Kepler), $n=3$ and $n\ge4$ in turns.
For the Kepler problem we show how to extend the $SO(3)$ group action
on $\R^3$ to an action of a subgroup of $SO(4)$ on quaternions. 
Treating 3 particles amounts to redoing
this construction for 3 difference vectors, and we show that for a
suitable chosen group action the space of quadratic invariants is
closed and the Hamiltonian can be written in terms of the quadratic
invariants. The corresponding Lie-Poisson structure is $\u(3,3)$.
In the final section we briefly comment on how this reduction is done
for an arbitrary number of particles.

\section{Simultaneous regularisation of binary collisions}

Let the positions of the $n$ particles be denoted by $\bm{q}_{i}\in\R^{3}$,
and the conjugate momenta by $\bm{p}_{i}\in\R^{3}$, $i=1,\dots,n$.
The translational symmetry is reduced by forming difference vectors
$\bm{q}_{ij}=\bm{q}_{i}-\bm{q}_{j}$ and $\bm{p}_{ij}=\bm{p}_{i}-\bm{p}_{j}$.
We follow \cite{Waldvogel} in using quaternions for the regularisation.
The analogue of Levi-Civita's squaring map can then be written as
\begin{equation}
\bm{q}=\bm{Q}*\bm{Q}^{\star},\label{eqn:qofQ}
\end{equation}
where $\bm{Q}=Q_{0}+\bm{i}Q_{1}+\bm{j}Q_{2}+\bm{k}Q_{3}$ and the
superscript $^{\star}$ flips the sign of the $\bm{k}$-component,
$\bm{Q}^{\star}=Q_{0}+\bm{i}Q_{1}+\bm{j}Q_{2}-\bm{k}Q_{3}$, see \cite{Waldvogel}.
By construction the quaternion $\bm{Q}*\bm{Q}^{\star}$ has vanishing
$k$-component and can thus be identified with the 3-dimensional vector
$\bm{q}$.

The mapping from 4-dimensional momenta $\bm{P}$ to 3-dimensional
momenta $\bm{p}$ is given by 
\begin{equation}
\bm{p}=\frac{1}{2||\bm{Q}||^{2}}\bm{Q}*\bm{P}^{\star}=\frac{1}{2}\bm{P}^{\star}*\bar{\bm{Q}}^{-1}\label{eqn:pofQP}
\end{equation}
where the overbar denotes quaternionic conjugation, i.e.{} flipping the sign of the $\bm i$, $\bm j$, and $\bm k$ component. 
%\footnote{the analogy with dividing by $Q$ is still weak, since it must be the conjugate of $\bm{Q}$...% } %
%\footnote{the 1/2 is natural, see Lemma below%} 
Note that in general the $\bm{k}$-component of the right hand side
is non-zero. One could think of the map to $\bm{p} \in \R^3$ to be a projection
onto the first three components. However, it turns out to be better
to impose that the last component vanishes. This condition can be
written as 
\begin{equation}
\bm{Q}^{T}K\bm{P}=0\quad\text{where}\quad K=\begin{pmatrix}0 & 0 & 0 & -1\\
0 & 0 & 1 & 0\\
0 & -1 & 0 & 0\\
1 & 0 & 0 & 0
\end{pmatrix}\,.\label{eqn:bilin}
\end{equation}
Here and in the following $\bm{Q}$ and $\bm{P}$ are interpreted as ordinary 4-dimensional
vectors; multiplication of quaternions by contrast is denoted by $*$.
Equation \eqref{eqn:bilin} is the famous bi-linear relation~\cite{KS}. Together \eqref{eqn:qofQ}
and \eqref{eqn:pofQP} define a projection $\pi$ from $(\bm{Q},\bm{P})\in T^{*}\R^{4}$
to $(\bm{q},\bm{p})\in T^{*}\R^{3}$. Only when restricting to the
subspace defined by the bi-linear relation \eqref{eqn:bilin} does
the map $\pi$ respect the symplectic structures so that
\[
\{f,g\}_{3}\circ\pi=\{f\circ\pi,g\circ\pi\}_{4}\,.
\]
Here the two Poisson brackets $\{,\}_{3}$ and $\{,\}_{4}$ are corresponding
to the two standard symplectic structures $d\bm{q}\wedge d\bm{p}$
and $d\bm{Q}\wedge d\bm{P}$, respectively,
see, e.g.~\cite{Kummer82,Iwai81}.
%\footnote{CHECK: I suppose this must be know, can we find a reference to this, or to
%some kind of symplectic version of it? E.g. Kummer82, Iwai81?%
%} %
%\footnote{another way of looking at this is to say that the reduction by the
%symmetry generated by the bi-linear relation gives a symplectic manifold
%$I^{-1}(0)/U(1)$.}

Using the transformation to $\bm Q, \bm P$ the Hamiltonian of the $n$-body problem
written in terms of difference vectors and scaling time gives the
regularised Hamiltonian \cite{Heggie} %
%\footnote{It should be noted that adding additional degrees of freedom to a
%Hamiltonian is by no means unique. Away from collision the dynamics
%of the inflated systems needs to project back to the original dynamics.
%}
%
%
\begin{align}
H=\; & \frac{1}{8}\left(\frac{R_{12}R_{31}}{\mu_{23}}\bm{P}_{23}^{T}\bm{P}_{23}+\frac{R_{12}R_{23}}{\mu_{31}}\bm{P}_{31}^{T}\bm{P}_{31}+\frac{R_{23}R_{31}}{\mu_{12}}\bm{P}_{12}^{T}\bm{P}_{12}\right)\label{eq:Hamiltonian}\\
 & -\frac{1}{4}\left(
 \frac{R_{23}}{m_{1}}(\bm{Q}_{31}\ast\bm{P}_{31}^{\star})^{T}(\bm{Q}_{12}\ast\bm{P}_{12}^{\star})+
 \frac{R_{31}}{m_{2}}(\bm{Q}_{12}\ast\bm{P}_{12}^{\star})^{T}(\bm{Q}_{23}\ast\bm{P}_{23}^{\star})+
 \frac{R_{12}}{m_{3}}(\bm{Q}_{23}\ast\bm{P}_{23}^{\star})^{T}(\bm{Q}_{31}\ast\bm{P}_{31}^{\star})
 \right)\nonumber \\
 & -(m_{2}m_{3}R_{31}R_{12}+m_{3}m_{1}R_{12}R_{23}+m_{1}m_{2}R_{23}R_{31})-hR_{23}R_{31}R_{12}.\nonumber 
\end{align}

where $R_{ij}=\bm{Q}_{ij}^{T}\bm{Q}_{ij}=\norm{\bm{{q}}_{ij}}$ and
$\mu_{ij}=\dfrac{m_{i}m_{j}}{m_{i}+m_{j}}$ is the reduced mass of
particles $i$ and $j$.

%Our Hamiltonian isn't exactly like Heggie because he absorbed a 2 into his A matrix... see comment on the footnote above and KSquaternion Hamiltonian - Diana%
%And I think we should use \star instead cause * is usually used for conventional conjugate%

\section{The Kepler Problem $n=2$}

As mentioned in the introduction this case has been treated extensively in the literature
\cite{KS,Iwai81,Kummer82}, but we briefly treat it first to establish our notation and important Lemmas needed 
for the case with 3 or more masses.
For $n=2$ there is only a single difference vector $\bm{q}_{12}=\bm{q}_{1}-\bm{q}_{2}$,
similarly for $\bm{p}$. For ease of notation, in this section we
are writing $\bm{q}$ for $\bm{q}_{12}$, similarly for $\bm{p}$,
and the corresponding quaternions $\bm{Q}$ and $\bm{P}$.

For $n=2$ the time scaling reduces the problem to the harmonic oscillator since the 
more complicated cross-terms in the kinetic energy vanish, so that
\[
    H = \frac{1}{8\mu} |\bm{P}|^2  - m_1 m_2 - h |\bm{Q}|^2 \,.
\]

The $SO(3)$ symmetry acting on pairs of difference vectors in $\R^{3}\times\R^{3}$
is the diagonal action $\Phi_{R}:(\bm{q},\bm{p})\mapsto(R\bm{q},R\bm{p})$
for $R\in SO(3)$. This is a symplectic map whose momentum map is
the cross product $\bm{q}\times\bm{p}$. Which linear symplectic action
$\Psi_{S}$ of (a subgroup of) $SO(4)$ acting on $\R^{4}\times\R^{4}$
projects to $\Phi_{R}$ under $\pi$?
\begin{lem}
The diagonal action $\Psi_{S}:(\bm{Q},\bm{P})\mapsto(S\bm{Q},S\bm{P})$
for $S\in G$ a subgroup of $SO(4)$ with $G\cong SU(2)\times SO(2)$  
projects to the action of $\Phi_{R}$ under $\pi$. In other words,
the diagram 
\[
\begin{CD}T^{*}\R^{3}@<\pi<<T^{*}\R^{4}\\
@V\Phi_{R}VV@V\Psi_{S}VV\\
T^{*}\R^{3}@<\pi<<T^{*}\R^{4}
\end{CD}
\]
commutes. \end{lem}
\begin{proof}
Let the rotation $R\in SO(3)$ be given by $R=\exp{At}$ for some
$A\in\Skew(3)$. We assume that 
% $G$ is a topologically closed subgroup of $SO(4)$ so that 
we can write $S=\exp{Bt}$ for $B\in\Skew(4)$.
The diagram states that $\Phi_{R}\circ\pi=\pi\circ\Psi_{S}$. Linearising
at the identity, i.e.\ differentiating with respect to $t$ and setting
$t=0$, and using that $\Phi_{S}$ leaves the norm of quaternions
unchanged gives 
\[
A(\bm{Q}*\bm{P}^{\star})=(B\bm{Q})*\bm{P}^{\star}+\bm{Q}*(B\bm{P})^{\star}
\]
from the momenta \eqref{eqn:pofQP}, and the same equation with $\bm{P}$
replaced by $\bm{Q}$ from the positions \eqref{eqn:qofQ}. For 
given $A=\hat{\bm{L}}$ with arbitrary $\bm{L}=(L_{x},L_{y},L_{z})^{t}$ and the usual hat-map
from $\R^{3}$ to $\Skew(3)$, the general solution can be written
as $B=\frac{1}{2}(\Isoc(\hat{\bm{L}})+\tau K)$, where $\Isoc(\hat{\bm{L}})=\begin{pmatrix}\hat{\bm{L}} & -\bm{L}\\
\bm{L}^{t} & 0
\end{pmatrix}$ and parameter $\tau$. The subgroup $G$ contains the subgroup of isoclinic rotations
$\exp(\Isoc(A))=\cos\omega\, I_{4}+\omega^{-1}\sin\omega\Isoc(A)$
where $\omega^{2}=\frac{1}{2}\Tr AA^{t}$. They form a subgroup since
the corresponding generators $\Isoc(A)$ form an algebra with 
$[\Isoc(\hat{\bm{a}}),\Isoc(\hat{\bm{b}})]=2\Isoc([\hat{\bm{a}},\hat{\bm{b}}])=2\Isoc(\widehat{\bm{a}\times \bm{b}})$
for any $\bm{a}, \bm{b} \in \R^3$.
% $\Isoc(A)^2 = -\omega^2 I_4$ and
% $\Isoc(\hat x) \Isoc(\hat y) = -(x, y) I_4 + \Isoc( [ \hat x , \hat y]  )$ and $  [ \hat x , \hat y] = \widehat{ x \times y}$
% where $(x,y)$ is the Euclidean scalar product and $x \times y$ the cross product.
The corresponding group of left-isoclinic rotation matrices $\exp(\Isoc(\hat{\bm{a}}))$
has a composition law given by left-multiplication of unit quaternions
with imaginary part proportional to $\bm{a}$.
% , and hence is isomorphic to $S^{3}\cong SU(2)$. 
The whole group $G$ is obtained by multiplying
the general left-isoclinic rotation $\exp(\Isoc(A))$ with the
special right-isoclinic rotation $\exp(K\tau)$. These two commute,
since $\Isoc(A)$ and $K$ commute. The group $\exp(K\tau)$ is isomorphic
to $SO(2)$, so $G$ is isomorphic of $SO(3) \times SO(2)$. 
%
%\footnote{It should be $ SO $(3) because the left isoclinic rotations contain the central inversion but this is also contained in the right isoclinic rotations SO(2) so we have to quotient that out first. This is not apparent if we just look at the algebra because in this case the exponential map is not actually injective. }
%
\end{proof}

For completeness we  now briefly mention the momentum map of $\Psi_S$, which was found (in different disguise) 
by Kummer in \cite{Kummer82}.
Here $\Im$ applied to a quaternion $ a+\bm{i}b+\bm{j}c +\bm{k}d $ produces the 3-dimensional vector $ (a,b,c)^t $. 
\begin{lem}
The group action $\Psi_{S}$ has momenta $\bm{L}=\Im(\frac{1}{2}\bm{Q}*\bar{\bm{P}}*\bm{k})$
and $L_{\tau}=\bm{Q}^{T}K\bm{P}$ which are mapped into the Lie algebra
$\mathfrak{g}$ of $G$ by $\frac{1}{2}(\Isoc(\hat{\bm{L}})+ K L_\tau)$. If in
addition the bilinear relation is imposed, then $\pi\circ \bm{L}$ becomes
the ordinary momentum $\bm{q}\times\bm{p}$. 
\end{lem}
\begin{proof}
The ODE whose flow is $\Psi_S$ with $S = \exp( B t)$ is $\dot {\bm{Q}} = B \bm{Q}$, $\dot{ \bm{P}} = B \bm{P}$,
which comes from the Hamiltonian ${\bm Q}^T B \bm{P}$. 
Three components of the angular momentum are thus
$L_{\alpha}=\frac{1}{2}\bm{Q}^{T}\Isoc(\hat{\bm{e}}_\alpha)\bm{P}$ where $\alpha\in\{x,y,z\}$
and $\bm{e}_\alpha$ is the unit vector in the direction of $\alpha$.
These components indeed form the first three components of $\frac{1}{2}\bm{Q}*\bar{\bm{P}}*\bm{k}$,
for example, $\bm{Q}^{T}\Isoc(\hat{\bm{e}}_y)\bm{P}=\frac{1}{2}(Q_{1}P_{3}-P_{1}Q_{3}+Q_{2}P_{4}-P_{2}Q_{4})$
is the $\bm{i}$-component of the quaternion $\frac{1}{2}\bm{Q}*\bar{\bm{P}}*\bm{k}$. 

The second statement is shown through direct computation. See \cite{Kummer82}
for more details.\end{proof}

In this section, the final simple but crucial observation is that $\Psi_S$ has four simple quadratic invariant polynomials,
which will form the new coordinates in the singular reduction.
\begin{lem}
The basic polynomial invariants of the group action $\Psi_{S}$ of
$G$ are 
\[
X_{1}=\bm{Q}^{T}\bm{Q}/\sqrt{2},\quad X_{2}=\bm{P}^{T}\bm{P}/\sqrt{2},\quad X_{3}=\bm{Q}^{T}\bm{P},\quad X_{4}=\bm{P}^{T}K\bm{Q}\,.
\]
The Poisson bracket of these invariants is closed and is the Lie-Poisson structure of $\mathfrak{u}(1,1)$. \end{lem}
\begin{proof}
Firstly, $SO(4)$, as the group of rotations preserves the inner product
on $\R^{4}$. Thus, $G$ as subgroup of $SO(4)$ must also preserve the
inner product. Hence, $X_1$, $X_2$, $X_3$ are clearly invariants. 

Since $\psi_S$ acts in the same way on $\bm{Q}$ and $\bm{P}$, in order to find
invariant quadratic forms it is enough to consider invariant forms 
$\bm{U}^T M \bm{V}  = ( S \bm U)^T S M \bm V  = \bm U^T S^t M S \bm V$ for arbitrary vectors $\bm U, \bm V$.
If $M = S^t M S$  holds for some $M$ then there are three invariant quadratic forms
given by $\bm Q^T M \bm Q$, $\bm Q^T M \bm P$, and $\bm P^T M \bm P$.

Now $S = \exp(B t)$, and differentiating at 0 implies that $BM = MB$. 
This has only two independent solutions, $M = I$ and $M = K$.
The antisymmetric $M=K$ only produces one non-zero invariant $X_4 = \bm Q^T K \bm P$, 
and $M=I$ reproduces the three scalar product invariants already mentioned.
Equivalently, one can show by direct computation that the only quadratic forms
that simultaneously have vanishing Poisson bracket with  $L_{\alpha}$ for all $\alpha\in\{x,y,z,\tau\}$ 
are in the span of $X_1, \dots, X_4$.
Therefore, the set $X_1, \dots, X_4$ is a basis for the vector
space of quadratic invariants. The only non-vanishing Poisson brackets are
\[
 \{ X_1, X_2 \} = 2 X_3, \quad 
 \{ X_2, X_3 \} = - 2 X_2, \quad
 \{ X_3, X_1 \} = - 2 X_1 \,.
\]
Clearly the invariants $X_1, \dots, X_4$ are closed under the Poisson bracket. 

Using $(X_{1},X_{2},X_{3},X_{4})$ as a basis for the space of quadratic invariants, 
the Poisson structure matrix is 
\[
\left(\begin{array}{rrrr}
0 & 2X_{3} & 2X_{1} & 0\\
-2X_{3} & 0 & -2X_{2} & 0\\
-2X_{1} & 2X_{2} & 0 & 0\\
0 & 0 & 0 & 0
\end{array}\right)\,,
\]
with Casimir $X_4$, the bi-linear relation.

The algebra $\mathfrak{u}(1,1; i J_2)$ is the set of complex matrices $M$ that satisfy 
$(H_1M)^\dagger + H_1M = 0$ for the hermitian matrix $H_1 = iJ_2$, where $J_2$ is the standard 
symplectic $2\times 2$ matrix, so that $H_1$ has eigenvalues $1, -1$, and hence signature $(1,1)$.
If we chose 
\[
   b_1  = \begin{pmatrix} 0 & 0 \\ -\sqrt{2} & 0 \end{pmatrix}, \quad 
   b_2  = \begin{pmatrix} 0 & \sqrt{2} \\  0 & 0 \end{pmatrix}, \quad 
   b_3  = \begin{pmatrix} 1 & 0 \\  0 & -1 \end{pmatrix}, \quad 
   b_4  = -i I \\
\]
as a basis for $\mathfrak{u}(1,1)$, then the algebra of commutators of $b_i$ is identical
to the algebra of Poisson brackets of the $X_i$.
\end{proof}
A relation of the regularised Kepler problem to $\mathfrak{u}(2,2)$ can be found in \cite{Kummer82}, 
but this is not directly related to our $\mathfrak{u}(1,1)$, which is the algebra of the quadratic invariants. 
Note that  $X_4$ is not only an  invariant but also a generator. 
The other generators are the component of $L$, and are not invariant under $\Psi_S$.
However, the sum of their squares is, and hence can be written in terms of the
above invariants: $L_x^2 + L_y^2 + L_z^2 = \frac{1}{2}X_1X_2-\frac{1}{4}X_3^2$.

The normalisation of the basis and the invariants is chosen so that basis vectors are normalised
with respect to the scalar product $\langle A, B \rangle = \Tr ( A^\dagger B) /2$. This ensures that 
the Lax form of the equations $\dot L = [ P, L]$, which we are now going to derive, is particularly symmetric.
The Hamiltonian in terms of the quadratic invariants is a linear function 
\[
    H = \frac{1}{8\mu} X_2  - m_1 m_2 - h X_1 \,.
\]
and the equations of motion are linear as well and given by
\begin{equation} \label{eqn:KeplerX}
\begin{pmatrix} \dot X_1 \\ \dot X_2 \\ \dot X_3 \end{pmatrix}
=
\left(\begin{array}{rrrr}
0 & 2X_{3} & 2X_{1} \\
-2X_{3} & 0 & -2X_{2} \\
-2X_{1} & 2X_{2} & 0 
\end{array}\right)
\begin{pmatrix}
H_1 \\ H_2 \\ H_3
\end{pmatrix} \,,
\end{equation}
where $H_i = \partial H / \partial X_i$.
Since $X_4$ commutes with all $X_i$ we can ignore it. On the level of the algebra we reduce 
by the centre, and get $\mathfrak{su}(1,1)$.
To emphasise the $\mathfrak{su}(1,1; iJ_1) = \mathfrak{sl}(2, \R)$ structure these equations can be written
in Lax form by defining 
\begin{equation} \label{eqn:Lax2}
%    L = \begin{pmatrix} X_3 & \sqrt{2} X_2 \\ -\sqrt{2} X_1 & -X_3 \end{pmatrix}
    L = J_2 \begin{pmatrix} \sqrt{2} X_1 &  X_3 \\  X_3 &  \sqrt{2} X_2 \end{pmatrix}
\quad \text{and} \quad
%   P = \begin{pmatrix}  -H_3 & \sqrt{2} H_1 \\ -\sqrt{2} H_2 & H_3 \end{pmatrix} 
   P = \begin{pmatrix} \sqrt{2} H_1 & H_3   \\  H_3 & \sqrt{2} H_2 & \end{pmatrix} J_2
\end{equation}
where $J_2$ is the standard symplectic $2\times 2$ matrix, 
so that the equations of motion \eqref{eqn:KeplerX} are equivalent to
\[
    \dot L = [ P, L ] \,.
\]
This is yet another way to write the regularised equations of the Kepler problem.
We recover the angular momentum as the Casimir $\det L = 2 X_1 X_2 - X_3^2$.
In the case of 3 or more bodies the Lax form of the equation gives non-trivial 
additional information on the Casimirs, see below.

Notice that  the symmetry reduction using invariants of the un-regularised 
Kepler problem leads to a Poisson structure of $\mathfrak{ sl}(2, \R)$ as well, 
see, e.g.~\cite{Dullin13},
however, with a different (non-regularised) Hamiltonian.

\section{The 3-body problem, $n=3$}

The $G$-action $\Psi_{S}$ on pairs $(\bm{q},\bm{p})$ extends to an action
(denoted by the same letter) on triples of pairs $(\bm{q}_{ij},\bm{p}_{ij})$.
Since the action acts diagonally, to get the corresponding angular momenta
the individual momenta are simply added together, $\mathcal{L}_{a}=\sum L_{a}^{i}$
for $a\in\{x,y,z\}$. In this way the action projects down  by $\pi$ to the 
usual action of the angular momentum.

Choosing the correct symmetry group is crucial in order to obtain
a good set of quadratic invariants. Since any flow generated by $L_\tau^i$ is
annihilated by $\pi$ there is a choice in defining the symmetry group and its action.
We could define an action of $SU(2) \times SO(2) \times SO(2) \times SO(2)$
where the action of each $SO(2)$ is the flow generated by $L_\tau^i$, $i= 1, 2, 3$
for the three particles. The set of quadratic invariants is then much smaller, since 
the group is larger. 
However, Heggie's Hamiltonian cannot be written in terms of these 9 quadratic invariants, 
even though it is clearly invariant under it. Instead of working with higher degree invariants,
we prefer to stick to quadratic invariants and instead consider a smaller group action.
Hence we keep the same group $G = SU(2) \times SO(2)$ and let 
$SO(2)$ act diagonally on the three particles. The corresponding flow is generated by 
$\mathcal{L}_\tau = \sum L_\tau^i$.
With this choice of extended $G$-action $\Psi_{S}$ gives the smallest set of closed quadratic
invariants in terms of which the Hamiltonian can be expressed.
\begin{lem}
\label{lem:Q} A quadratic form $Q= \bm X^T M \bm X$ that is invariant under $\Psi_{S}$ has matrix
\[
M=[W]_\mathrm{{sym}}\otimes I_{4}+[W]_\mathrm{{skew}}\otimes K
\]
where $W$ is an arbitrary $6\times6$ matrix, $\bm X=(\bm{Q}_{1}^{T},\bm{Q}_{2}^{T},\bm{Q}_{3}^{T},\bm{P}_{1}^{T},\bm{P}_{2}^{T},\bm{P}_{3}^{T})^T$
and $\otimes$ denotes the Kronecker product. The vector space of quadratic
invariants of this form is closed under the Poisson bracket. % and hence form a Lie algebra $\mathfrak{g}$.
\end{lem}
\begin{proof}
Since $\Psi_S$ acts diagonally, the arguments from Lemma 3.3 can be repeated.
Hence the invariant quadratic forms are either of the form 
$\alpha_{ij} = \bm Q_i^T \bm Q_j g_{ij}$,
$\beta_{ij} = \bm P_i^T \bm P_j g_{ij}$,
$\gamma_{ij} = \bm Q_i^T \bm P_j$,
where $g_{ij} = 1/\sqrt{2}$ for $i  = j$ and $1$ otherwise,
or they involve the matrix $K$ and are 
$a_{ij} = \bm Q_i^T K \bm Q_j$,
$b_{ij} = \bm P_i^T K \bm P_j$,
$c_{ij} = \bm Q_i^T K \bm P_j$.
The first group has 21 elements,
and the second group has 15 elements because
the expressions are identically zero when $i=j$.
Any quadratic form on phase space can be written as $\bm X^T M \bm X$.
For the first group of 24 invariant quadratic forms we have  $ M = S \otimes I_{4}$ where $S \in \mathrm{sym(6)}$. 
Similarly, for the second group of quadratic forms over $K$ we have  $M = A \otimes K$ where $A \in \mathrm{skew(6)}$.
As the sum of invariants is invariant, the matrix for any quadratic invariant can be written
as $S\otimes I_{4}+A\otimes K$.
Thus, the set of quadratic invariants is of the form $Q=\bm X^T M \bm X$ where
\begin{equation} \label{eqn:MWform}
  2 M = [W]_\mathrm{{sym}}\otimes I_{4}+[W]_\mathrm{{skew}}\otimes K
\end{equation}
where $W$ is an arbitrary $6\times6$ matrix and so the space of
quadratic invariants is isomorphic to $Mat(6\times6,\R)$ as a vector
space. 
Let $\bm X^T M \bm X$, and $\bm X^T N\bm X$ be two arbitrary quadratic forms.
Then the Poisson bracket induces an algebra on symmetric matrices given by 
\begin{equation} \label{eqn:PBalg}
      M * N = 2 [ M J N]_\mathrm{{sym}} = M J N - N  J M \,.
\end{equation}
It is well known that for general symmetric matrices this algebra is $\mathfrak{sp}(m)$
where $m = {\rm dim}( \bm X)$. 
In our case we have a sub-algebra of matrices of the form \eqref{eqn:MWform}, say
$2 M = \tilde{A}\otimes I_{4}+\check{A}\otimes K$, 
$2 N = \tilde{B}\otimes I_{4}+\check{B}\otimes K$, 
and using $J = J_6 \otimes I_4$, where where $ \tilde{(\cdot)} =[(\cdot)]_\text{sym} $ and  $\check{(\cdot)} =[(\cdot)]_\text{skew} $ we find
\begin{align}
2  M*N & =[(\tilde{A}\otimes I_{4}+\check{A}\otimes K)(J_6 \otimes I_{4})(\tilde{B}\otimes I_{4}+\check{B}\otimes K)]_\mathrm{{sym}}\\
 & =[(\tilde{A}J_6 \tilde{B}-\check{A}J_6 \check{B})\otimes I_{4}+(\check{A}J_6\tilde{B}+\tilde{A}J_6\check{B})\otimes K]_\mathrm{{sym}}\\
% & =(\tilde{A}J_6\tilde{B}-\tilde{B}J_6\tilde{A}-\check{A}J_6\check{B}+\check{B}J_6\check{A})\otimes I_{4}+
 %     (\check{A}J_6\tilde{B}-\tilde{B}J_6\check{A}+\tilde{A}J_6\check{B}-\check{B}J_6\tilde{A})\otimes K\\
 & =  [ \tilde A J_6 \tilde B - \check A J_6 \check B ]_\mathrm{{sym}} \otimes I_4 +  [ \check A J_6 \tilde B + \tilde{A}J_6\check{B} ]_\mathrm{{skew}} \otimes K \label{eqn:Kron}
%  & =-J([J\tilde{A},J\tilde{B}]-[J\check{A},J\check{B}])\otimes I_{4}-J([J\tilde{A},J\check{B}]+[J\check{A},J\tilde{B}])\otimes K
\end{align}
so that this sub-algebra, and hence the Poisson bracket of quadratic invariants of the form \eqref{eqn:MWform}, is closed.
Note that the Kronecker product of two antisymmetric matrices is symmetric.
As particular examples of the above general rule we have, e.g., that 
$\{\alpha_{1,1},\beta_{1,1}\}=2\gamma_{1,1}$, and $\{\alpha_{1,1},c_{3,1}\}=-\sqrt{2} a_{1,3}$ .
By setting the antisymmetric parts $\tilde A$ and $\tilde B$ to zero, it is clear  
that the 21-dimensional subspace spanned by 
$\alpha, \beta, \gamma$ is closed under the Poisson bracket and hence forms a sub-algebra
within the sub-algebra of invariant quadratic forms.
\end{proof}

Defining $f_{ij}=4(\gamma_{i,j}\gamma_{j,i}-\gamma_{i,i}\gamma_{j,j}+\beta_{i,j}\alpha_{i,j}-c_{i,j}c_{j,i}+b_{i,j}a_{i,j})$,
the Hamiltonian in terms of the invariant quadratic forms reads 
\begin{align*}
H & =\frac{1}{8}\left(\frac{\alpha_{2,2}\alpha_{3,3}}{\mu_{23}}\beta_{1,1}+\frac{\alpha_{3,3}\alpha_{1,1}}{\mu_{13}}\beta_{2,2}+\frac{\alpha_{1,1}\alpha_{2,2}}{\mu_{12}}\beta_{3,3}\right)\\
 & -\frac{1}{16}\left(\frac{\alpha_{1,1}}{m_{1}}f_{23}+\frac{\alpha_{2,2}}{m_{2}}f_{13}+\frac{\alpha_{3,3}}{m_{3}}f_{12}\right)\\
 & -m_{2}m_{3}\alpha_{2,2}\alpha_{3,3}-m_{1}m_{3}\alpha_{1,1}\alpha_{3,3}-m_{1}m_{2}\alpha_{1,1}\alpha_{2,2}-h\alpha_{1,1}\alpha_{2,2}\alpha_{3,3}\,.
\end{align*}
Using this Hamiltonian we can now write down the regularised symmetry reduced 3-body dynamics as 
$ \dot f = \{ f, H \}$, where $f$ is any function of the 36 invariants.
It is rather unfortunate that the dimension of the space of invariants is bigger than the 
dimension of the original phase space, so from the point of view of efficiency of 
numerical integration nothing can be gained here.

In order to work out the isomorphism type of the Lie algebra of
quadratic invariants, we first induce a Lie bracket $[\cdot ,\cdot]_m$ on $\mathrm{Mat}(6\times6,\R)$ 
using the Poisson bracket. We then show that this bracket is 
isomorphic to $\mathfrak{u}(3,3)$ with the standard commutator bracket.

The Lie-algebra of quadratic invariants (respectively of their symmetric matrices) defined in 
\eqref{eqn:PBalg} induces a Lie-algebra on $\mathrm{Mat}(6\times 6, \R)$

simply by reading off the first factors of the Kronecker product in \eqref{eqn:Kron}, thus we define
\begin{equation} \label{eqn:defnbram}
   [\tilde A + \check A, \tilde B + \check B]_m = 
     2  [ \tilde A J_6 \tilde B - \check A J_6 \check B ]_\mathrm{{sym}} +  2 [ \check A J_6 \tilde B + \tilde{A}J_6\check{B} ]_\mathrm{{skew} } \,.
\end{equation}

This leads us to the main theorem of this paper:
\begin{thm}
\label{th:1} The symmetry reduced regularised 3-body problem has
a Lie-Poisson structure with algebra $\mathfrak{u}(3,3)$ and a corresponding
Hilbert basis of 36 quadratic functions invariant under $\Psi_{S}$. \end{thm}
\begin{proof}
For $A\in \mathrm{Mat}(6\times6,\R)$  the matrix $M=J(\tilde{A}+i\check{A})$
is in $\mathfrak{u}(3,3)$. Here the indefinite Hermitian algebra is defined 
with respect to the indefinite Hermitian matrix $H=iJ$ 
with eigenvalues $\pm1$ each with multiplicity 3 so that the signature is $(3,3)$.
Now it is easy to check that $(HM)^{\dagger}+HM=0$, and that 
matrices of the form $J(\tilde{A}+i\check{A})$ are closed under the commutator, 
and hence are in $\mathfrak{u}(3,3)$.

Now we show that the vector space isomorphism $h:\mathrm{Mat}(6\times6,\R)\to\mathfrak{u}(3,3)$
with $h(A)= J_6(\tilde{A}+i\check{A})$
 is in fact an isomorphism of Lie algebras: 
$h( [A,B]_m) = [h(A), h(B)]$.
First notice that the bracket $[,]_m$ from \eqref{eqn:defnbram} can be rewritten as
\[
  [A,B]_m =  [\tilde A + \check A, \tilde B + \check B]_m = 
     -J_6 ([J_6 \tilde{A},J_6 \tilde{B}]-[J_6\check{A},J_6\check{B}]+[J_6\tilde{A},J_6\check{B}]+[J_6\check{A},J_6\tilde{B}]) \,,
\]
so that 
$\left[[A,B]_{m}\right]_{sym}=-J_6([J_6\tilde{A},J_6\tilde{B}]-[J_6\check{A},J_6\check{B}])$ and 
$\left[[A,B]_{m}\right]_{skew}=-J_6([J_6\tilde{A},J_6\check{B}]+[J_6\check{A},J_6\tilde{B}])$.
Hence, on the one hand we have
\[
h( [A, B]_m) = J_6(  -J_6([J_6\tilde{A},J_6\tilde{B}]-[J_6\check{A},J_6\check{B}])  + i ( -J_6([J_6\tilde{A},J_6\check{B}]+[J_6\check{A},J_6\tilde{B}]) )    ) \,.
\]
On the other hand we have 
\[
[ h(A), h(B) ] =  [ J_6( \tilde A + i \check A), J_6 ( \tilde B + i \check B)]
\]
and expanding the commutator and collecting real and imaginary parts shows that in deed this equals $h([A, B]_m)$.
This proves that space of quadratic invariants and $\mathfrak{u}(3,3)$
are isomorphic as Lie algebras.
\end{proof}

In the final step we use the isomorphism just established to write the equations
of motion in Lax form $\dot L = [ P, L]$, using the Lie-Poisson bracket on $\mathfrak{u}(3,3)$. 
This brings out most clearly the Casimirs of the reduced system, which are the traces
of powers of $L$, or, alternatively, the coefficients of the characteristic polynomial of $L$.
Note that the Lax form gives only 6 invariants, but since there are 36 variables the system is 
by no means integrable.

\begin{lem}
The Poisson structure has 6 Casimirs of degree 1 through 6. The linear
Casimir is the sum of the bilinear integrals $\mathcal{L}_{\tau}$,
the quadratic Casimir is the sum of the three angular momenta squared
$\mathcal{L}_{x}^{2}+\mathcal{L}_{x}^{2}+\mathcal{L}_{z}^{2}$. \end{lem}
\begin{proof}
The Poisson bracket of the Lie algebra, in this matrix representation,
can be written as 
\[
\{f,g\}(M)=\left\langle M,\left[\frac{df}{dM},\frac{dg}{dM}\right]\right\rangle
\]
where, the inner product is given by $\langle M,N\rangle=\Tr(M^{\dagger}N)/2$ and
$\frac{df}{dM}$ refers to the element in $\mathfrak{g}$ that satisfies
\[
\lim_{\varepsilon \to 0}[f(M+\varepsilon \, dM)-f(M)]=\left\langle dM,\frac{df}{dM}\right\rangle,
\]
see, e.g.,  \cite{MarsdenAndRatiu} for more details. 
The reason for choosing a normalised basis is that with respect to a normalised basis 
this can be written in the simple form $\dot L = [ P, L]$, see below.
Now the co-effiecients of the characteristic polynomial of $L$ are in fact the Casimirs of the Poisson bracket. 
The co-efficient of the fifth order term is just the sum of the bilinear integrals,
$\mathcal{L}_{\tau}$. The coefficient of the quartic term is 
\[
\mathcal{L}_{x}^{2}+\mathcal{L}_{y}^{2}+\mathcal{L}_{z}^{2}+f(\mathcal{L}_{\tau})
\]
where $f(\mathcal{L}_{\tau})$ is a quadratic function of the bilinear
integrals. Under the reduction by the centre, this Casimir simply
becomes $\mathcal{L}_{x}^{2}+\mathcal{L}_{x}^{2}+\mathcal{L}_{z}^{2}$. 
\end{proof}

Define 
\[
 M=  \begin{pmatrix} 
    \sqrt{2} \alpha_{1,1} & \alpha_{1,2} + i a_{1,2} & \alpha_{1,3} + i a_{1,3} & \gamma_{1,1} + i c_{1,1} &  \gamma_{1,2} + i c_{1,2}  & \gamma_{1,3} + i c_{1,3} \\
    \alpha_{1,2} - i a_{1,2}  & \sqrt{2} \alpha_{2,2} & \alpha_{2,3} + i a_{2,3} & \gamma_{2,1} + i c_{2,1} &  \gamma_{2,2} + i c_{2,2}  & \gamma_{2,3} + i c_{2,3} \\
    \alpha_{1,3} - i a_{1,3} & \alpha_{2,3} - i a_{2,3} & \sqrt{2} \alpha_{3,3}  & \gamma_{3,1} + i c_{3,1} &  \gamma_{3,2} + i c_{3,2}  & \gamma_{3,3} + i c_{3,3} \\
    \gamma_{1,1} - i c_{1,1} &  \gamma_{2,1} - i c_{2,1}  & \gamma_{3,1} - i c_{3,1} & \sqrt{2} \beta_{1,1} & \beta_{1,2} + i b_{1,2} & \beta_{1,3} + i b_{1,3} \\
    \gamma_{1,2} - i c_{1,2} &  \gamma_{2,2} - i c_{2,2}  & \gamma_{3,2} - i c_{3,2} & \beta_{1,2} - i b_{1,2} & \sqrt{2} \beta_{2,2} & \beta_{2,3} + i b_{2,3} \\
    \gamma_{1,3} - i c_{1,3} &  \gamma_{2,3} - i c_{2,3}  & \gamma_{3,3} - i c_{3,3} & \beta_{1,3} - i b_{1,3} & \beta_{2,3} - i b_{2,3}  & \sqrt{2} \beta_{3,3} 
   \end{pmatrix}
\]
and 
\[
   L = J_6 M, \quad
    P =  dM  J_6,
\]
so that the equations of motion of the symmetry reduced regularised spatial 3-body problem can be written 
in Lax form
\[
    \dot L = [ P, L] \,.
\]
Because we have chosen a self-dual basis the matrix $dM$ is simply given by 
replacing each entry in $M$ by the derivative of the Hamiltonian $H$ with respect to the variables of that entry, 
compare \eqref{eqn:Lax2}.
Reduction by the centre of the algebra which is generated by the Linear Casimir $\mathcal{L}_{\tau}$ gives $\su(3,3)$.
This can be achieved by subtracting $\Tr L = - 2 i \sum c_{i,i}$ in the diagonal of $L$, but the equations are more symmetric 
if we stay in $\mathfrak{u}(3,3)$.

The fact that the three difference vectors $\bm{q}_{ij}$ add to zero
induces  three additional quadratic integrals $T_{1},T_{2},T_{3}$. The
flow of these integrals is non-compact, and we were not able to use
it for symmetry reduction. %
%\footnote{but also the differences in momenta add to zero, does this correspond
%to the three angular momenta, or are these additional integrals?} 
The three momenta $\mathcal L_x$,  $\mathcal L_y$,  $\mathcal L_z$, 
and the integrals $T_{i}$ form the Algebra $\mathfrak{se}(3)$.

\section{The $n$-body problem}
\begin{thm}
The symmetry reduced regularised $n$-body problem has a Lie-Poisson
structure with algebra $\mathfrak{u}(m,m)$ where $m=n(n-1)/2$.\end{thm}
\begin{proof}
As shown in Lemma \ref{lem:Q}, the nature of the invariants under
$\Psi_{S}$ are independent of the number of particles. They are realised
as in the aforementioned lemma in phase space by the use of symmetric
and antisymmetric matrices of size $2m\times2m$ where $m$ denotes
the number of difference vectors in the system. This establishes the
vector space isomorphism to the space of $2m\times2m$ matrices. Furthermore,
by Theorem \ref{th:1}, it is apparent that the Lie algebra of invariants
is isomorphic to $\mathfrak{u}(m,m)$. As $m$ is equal to $\binom{n}{2}=n(n-1)/{2}$,
the algebra of invariants for the symmetry reduced regularised $n$-body
problem has a Lie-Poisson structure with algebra $\mathfrak{u}(n(n-1)/2,n(n-1)/2)$. 
\end{proof}

\section{Conclusion}

In this paper, we have shown that the quadratic invariants invariants
of the regularised n-body problem are either inner products or quadratic
forms over the antisymmetric matrix $K$. These invariants form a
Lie-Poisson algebra that is isomorphic to the lie algebra $\mathfrak{u}(m,m)$
where $m=n(n-1)/2$ which is the algebra corresponding to the group
that preserves hermitian forms of signature $(m,m)$. The dimension
of this Lie Algebra is of order $n^{4}$. Thus the use of such an
algebra to obtain numerical solutions is improbable for large values
of n. Despite this, the isomorphism to $\mathfrak{u}(m,m)$ yields
a large amount of information about the rich structure of these invariants
and provides insight into the n-body problem.

\section{Acknowledgment}

Diana Nguyen and Suntharan Arunasalam acknowledge support through a 
2013/14 Vacation Research Scholarship from the Australian Mathematical Sciences Institute AMSI, 
during which this paper was started.

\end{document}